\documentclass{amsproc}

\renewcommand\labelenumi{(\roman{enumi})}
\usepackage{graphicx}
\usepackage{amsrefs}

\newtheorem{theorem}{Theorem}[section]
\newtheorem{lemma}[theorem]{Lemma}

\theoremstyle{definition}

\theoremstyle{remark}
\newtheorem{remark}[theorem]{Remark}

\numberwithin{equation}{section}

\begin{document}

\title[On the solvability of BVP for linear DAE]{On the solvability of boundary value problems for linear differential-algebraic equations with constant coefficients}

\author{Anar Assanova}
\address{Institute of Mathematics and Mathematical Modeling, Pushkin Str. 125, 050010, Almaty, Kazakhstan}
\curraddr{}
\email{assanova@math.kz}
\thanks{
  \includegraphics[height=7.0mm]{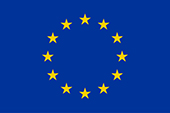} This paper is supported by European Union’s Horizon 2020 research and innovation programme under the Marie Sklodowska-Curie grant agreement ID: 873071, project SOMPATY (Spectral Optimization: From Mathematics to Physics and Advanced Technology)
}
\author{Carsten Trunk}
\address{Institut für Mathematik, Technische Universität Ilmenau, Postfach 100565, D-98684, Ilmenau, Germany}
\curraddr{}
\email{carsten.trunk@tu-ilmenau.de}
\author{Roza Uteshova}
\address{Institute of Mathematics and Mathematical Modeling, Pushkin Str. 125, 050010, Almaty, Kazakhstan}
\curraddr{}
\email{r.uteshova@iitu.edu.kz}
\thanks{}

\date{}
\begin{abstract}
We study a two-point boundary value problem for a linear differen\-tial-algebraic equation with constant coefficients by using the method of parameterization. The parameter is set as the value of the continuously differentiable component of the solution at the left endpoint of the interval. Applying the Weierstrass canonical form to the matrix pair associated with the differential-algebraic equation, we obtain a criterion for the unique solvability of the problem.
\end{abstract}

\maketitle


\section{Introduction}

We consider the linear differential-algebraic equation with constant coefficients of the form
\begin{equation}\label{dae}
E\dot{x}(t)=Ax(t)+f(t), \quad t\in (0,T),    
\end{equation}
subject to the boundary condition
\begin{equation}\label{BC}
Bx(0)+Cx(T)=d.    
\end{equation}
Here $E, A,B,C\in \mathbb{R}^{n\times n}$, $d\in\mathbb{R}^n$, $T>0$. We suppose that the matrix pair $(E,A)$ is regular, i.e. $\text{det}(\lambda E - A)\neq 0$ for some $\lambda\in\mathbb{C}$.

By a solution of the boundary value problem \eqref{dae}, \eqref{BC} we mean a function $x\in C^1([0,T],\mathbb{R}^n)$ satisfying equation \eqref{dae} and the boundary condition \eqref{BC}.

Differential-algebraic equations have become widespread over the last decades, being a tool for modeling and simulation of dynamical systems with constraints in numerous applications \cites{Kunkel:2006,brenan1995numerical,ascher1998computer,lamour2013differential, samoilenko2000linear,boichuk2004generalized,riaza2008differential}. The theory of boundary value problems for differential-algebraic equations started to develop by applying modified versions of the shooting and collocation methods designed for boundary value problems for ordinary differential equations \cites{marz1984difference,clark1989numerical,lamour1991well,lamour1997shooting,BAI1991269,bai1992modified,ascher1992projected,stover2001collocation,kunkel2002symmetric}. In \cite{amodio1997numerical},  P.\ Amodio and F.\ Mazzia  studied problem \eqref{dae}, \eqref{BC} by the method of boundary values. R.\ März applied the method of projectors and methods of perturbation theory \cite{marz1996canonical, marz2004solvability, marz2005characterizing} to problem \eqref{dae}, \eqref{BC}.  The monograph by R.\ Lamour, R.\ März, and C.\ Tischendorf  \cite{lamour2013differential}, devoted to the projector study of differential-algebraic equations, provides a detailed review in this field.

A series of papers by C. Trunk et al. \cite{gernandt2023characteristic, berger2016linear, leben2021finite, berger2021linear} investigate possibilities of generalization and extension of the Kronecker canonical form to differential-algebraic equations with rectangular matrices.  A number of methods and approaches have been developed for constructing their solutions. 

However, the methods developed for solving problem \eqref{dae}, \eqref{BC} may not always be applicable for a wide class of boundary conditions. As a consequence, this requires the development of new methods or modification of known methods of the theory of differential equations, which would be applicable to differential-algebraic equations and would be of constructive nature.

In this paper, we use the method of parameterization proposed by Dzhumabaev \cite{Dzhumabaev:1989}, which has proven to be an efficient constructive method allowing both to derive criteria for the unique solvability and obtain approximate solutions of various classes of boundary value problems \cite{dzhumabaev2010method, dzhumabaev2016one, dzhumabaev2018computational, asanova2013well, assanova2022solution}. This method was originally proposed for solving the linear boundary value problem \eqref{dae}, \eqref{BC} provided $\text{det}E\neq 0.$ In this case, a criterion for the unique solvability was obtained in terms of coefficients and an algorithm for approximate solution was developed.  

Our goal is to apply the method of parameterization  to the boundary value problem \eqref{dae}, \eqref{BC} in the case when the matrix $E$ is not necessarily non-singular. We derive a criterion for the existence of a unique solution under certain assumptions on the matrices of the boundary condition.

\section{Main results}
We introduce the parameter $\mu\in\mathbb{R}^n$ defined as $E\mu:=Ex(0)$. By substituting
\begin{equation*}
    u(t):=x(t)-\mu,
\end{equation*}
  equation \eqref{dae} is transformed into the initial value problem with parameter
\begin{equation}\label{ivp1}
E\dot{u}(t)=A(u(t)+\mu)+f(t),\quad t \in (0,T),
\end{equation}
\begin{equation}\label{ivp2}
Eu(0)=0.
\end{equation}
and the boundary condition \eqref{BC} becomes 
\begin{equation}\label{BC_u}
B(\mu+u(0))+C(\mu+u(T))=d.    
\end{equation}

The following statement shows the equivalence of the boundary value problem \eqref{dae}, \eqref{BC} and the boundary value problem with parameter \eqref{ivp1}-\eqref{BC_u}.

\begin{lemma}
If $x^{\ast}\in C^{1}([0,T],\mathbb{R}^n)$ is a solution of problem \eqref{dae}, \eqref{BC}, then the pair $(\mu^{\ast},u^{\ast})\in \mathbb{R}^n\times C^1([0,T],\mathbb{R}^n)$, where $E\mu^{\ast}=Ex^{\ast}(0)$ and $u^{\ast}(t)=x^{\ast}(t)-\mu^{\ast}$, is a solution of problem \eqref{ivp1}-\eqref{BC_u}. 

Conversely, if a pair $(\mu^{\ast\ast},u^{\ast\ast})\in \mathbb{R}^n\times C^1([0,T],\mathbb{R}^n)$ is a solution of problem \eqref{ivp1}-\eqref{BC_u}, then the function $x^{\ast\ast}\in C^{1}([0,T],\mathbb{R}^n)$, defined by $x^{\ast\ast}(t)=\mu^{\ast\ast}+u^{\ast\ast}(t)$, is a solution of problem \eqref{dae}, \eqref{BC}.
\end{lemma}

Let $P$ and $Q$ be non-singular matrices which transform \eqref{ivp1} and \eqref{ivp2} to Weierstrass canonical form \cite{Kunkel:2006}, i.e.,
\begin{equation}\label{WCF}
    PEQ=\begin{bmatrix}
I_{n_1} & 0\\
0 & N
\end{bmatrix},\quad PAQ=\begin{bmatrix}
J & 0\\
0 & I_{n_2}
\end{bmatrix},\quad Pf=\begin{bmatrix}
\tilde{f}_1\\
\tilde{f}_2
\end{bmatrix},
\end{equation}
where $J$ is an $n_1\times n_1$ matrix in Jordan canonical form and $N$ is an $n_2\times n_2$  nilpotent matrix also in Jordan canonical form; $n_1+n_2=n$. Following \cite{Kunkel:2006}, we call the index of nilpotency of $N$ in \eqref{WCF} the index of the matrix pair $(E,A)$, denoted by
\begin{equation*}
    \nu=\text{ind}(E,A).
\end{equation*}

According to the space decomposition given by \eqref{WCF}, we have the function 
\begin{equation*}
    \tilde{u}(t)=(\tilde{u}_1(t),\tilde{u}_2(t))^T :=Q^{-1}u(t)
\end{equation*}
and the vector 
\begin{equation*}
    \tilde{\mu}=(\tilde{\mu}_1,\tilde{\mu}_2)^T :=Q^{-1}\mu
\end{equation*}
with $\tilde{u}_j\in C^1([0,T],\mathbb{R}^{n_j})$ and  $\tilde{\mu}_j\in \mathbb{R}^{n_j}$, $j=1,2.$

We then can rewrite problem \eqref{ivp1}, \eqref{ivp2} in the following form:
\begin{equation}\label{difeq}
\dot{\tilde{u}}_1(t)=J(\tilde{u}_1(t)+\tilde{\mu}_1)+
\tilde{f}_1(t),
\end{equation}
\begin{equation}\label{difeq_ic}
\tilde{u}_1(0)=0,
\end{equation}
\begin{equation}\label{algeq}
N\dot{\tilde{u}}_2(t)=\tilde{u}_2(t)+\tilde{\mu}_2+
\tilde{f}_2(t),
\end{equation}
\begin{equation}\label{algeq_ic}
N\tilde{u}_2(0)=0.
\end{equation}

If we set $\tilde{B}:=BQ$ and $\tilde{C}:=CQ$, the boundary condition \eqref{BC_u} transforms into 
\begin{equation}\label{BC_trans}
\tilde{B}(\widetilde{\mu}+\widetilde{u}(0))+\tilde{C}(\tilde{\mu}+\tilde{u}(T))=d.    
\end{equation}

By a solution of the boundary value problem \eqref{difeq}-\eqref{BC_trans}  we mean a pair $(\tilde{\mu},\tilde{u})\in \mathbb{R}^n\times C^1([0,T],\mathbb{R}^n)$ with $\tilde{\mu}=(\tilde{\mu}_1,\tilde{\mu}_2)^T$ and $\tilde{u}(t)=(\tilde{u}_1(t),\tilde{u}_2(t))^T$ satisfying the initial value problems \eqref{difeq}, \eqref{difeq_ic} and \eqref{algeq}, \eqref{algeq_ic}, and the boundary condition \eqref{BC_trans}.

Before we start to study problem \eqref{difeq}-\eqref{BC_trans}, let us note that we are able to write down explicitly the solutions of the initial value problems \eqref{difeq},\eqref{difeq_ic} and \eqref{algeq},\eqref{algeq_ic}. 

\begin{lemma}\label{lemma_ivp_de}
    For any $\tilde{\mu}_1\in \mathbb{R}^{n_1}$, the initial value problem for the linear differential equation \eqref{difeq},\eqref{difeq_ic} has the unique solution 
\begin{equation}\label{u_dif}
  \tilde{u}_1(t)=J\int\limits_0^t e^{(t-s)J}ds\tilde{\mu}_1+\int\limits_0^t e^{(t-s)J}\tilde{f}_1(s) ds.
\end{equation}
\end{lemma}
\begin{proof}
The general solution of the ordinary differential equation \eqref{difeq} has the form \cite[p.\ 58]{hartman2002ordinary}
\begin{equation*}
    \tilde{u}_1(t)= e^{t J} \tilde{u}_1(0)+J\int\limits_0^t e^{(t-s)J}ds\tilde{\mu}_1+\int\limits_0^t e^{(t-s)J}\tilde{f}_1(s) ds.
\end{equation*}
Applying the initial condition \eqref{difeq_ic}, we obtain \eqref{u_dif}.
\end{proof}

By Lemma 2.8 \cite{Kunkel:2006}, equation \eqref{algeq} for fixed $\tilde{\mu}_2$ has the unique solution 
\begin{equation}\label{u_alg}
\tilde{u}_2(t)=-\sum\limits_{i=0}^{\nu-1}N^i[\tilde{\mu}_2+\tilde{f}_2(t)]^{(i)}=-\tilde{\mu}_2-\sum\limits_{i=0}^{\nu-1}N^i\tilde{f}_2^{(i)}(t),    
\end{equation}
 without specifying any initial values. Hence, taking into account \eqref{algeq_ic}, we obtain that the second component of the parameter $\tilde{\mu}$ is uniquely determined by
\begin{equation}\label{mu2}
    \tilde{\mu}_2=-\sum\limits_{i=0}^{\nu-1}N^i\tilde{f}_2^{(i)}(0).
\end{equation}
                                                            
Let us now turn to the boundary value problem with parameter \eqref{difeq}-\eqref{BC_trans}.  The second component $(\tilde{\mu}_2,\tilde{u}_2(t))$ of its solution $(\tilde{\mu},\tilde{u}(t))$ is already known from \eqref{u_alg} and \eqref{mu2}. So we are interested in finding only the first component $(\tilde{\mu}_1,\tilde{u}_1(t))$ with $\tilde{\mu}_1$ and $\tilde{u}_1(t)$ satisfying \eqref{u_dif}. Hence, the appropriate number of imposed boundary conditions must coincide with the number $n_1$ of differential equations in \eqref{dae}. So, it is reasonable to assume that the matrices $\tilde{B}, \tilde{C}\in \mathbb{R}^{n \times n}$ and the right-hand side $d\in\mathbb{R}^n$ of the boundary condition \eqref{BC_trans} are of the form
\begin{equation}\label{BC0}
    \tilde{B}=\begin{bmatrix}
\tilde{B}_1 & \tilde{B}_2\\0 & 0\end{bmatrix},\quad \tilde{C}=\begin{bmatrix}
\tilde{C}_1 & \tilde{C}_2\\
0 & 0
\end{bmatrix},\quad d=\begin{bmatrix}
d_1\\0\end{bmatrix},
\end{equation}
where $\tilde{B}_1,\tilde{C}_1\in \mathbb{R}^{n_1\times n_1}$, $\tilde{B}_2,\tilde{C}_2\in \mathbb{R}^{n_1\times n_2}$, $d_1\in \mathbb{R}^{n_1}$.

Now, inserting \eqref{u_dif} into \eqref{BC_trans} and taking into account Lemma \ref{lemma_ivp_de} and \eqref{u_alg}, we obtain the following algebraic equation in $\tilde{\mu}_1$:
\begin{equation}\label{alg_syst}
    \tilde{D}\tilde{\mu}_1=\tilde{d},
\end{equation}
where 
\begin{equation*}
\tilde{D}=\tilde{B}_1+\tilde{C}_1+\tilde{C}_1 J\int\limits_{0}^{T} e^{(T-s)J}ds    
\end{equation*}
and
\begin{equation*}
   \tilde{d}=d_1-\tilde{C}_1\int\limits_{0}^{T} e^{(T-s)J}\tilde{f}_1(s) ds+\tilde{B}_2\sum\limits_{i=0}^{\nu-1}N^i\tilde{f}_2^{(i)}(0)+\tilde{C}_2\sum\limits_{i=0}^{\nu-1}N^i\tilde{f}_2^{(i)}(T).
\end{equation*}

The equation  \eqref{alg_syst} has a unique solution if the matrix $\tilde{D}$ is non-singular. Let us assume, for instance, that $J=\text{diag}(\lambda_1,\ldots,\lambda_{n_1})$. In this case ,
\begin{equation*}
\tilde{D}=\tilde{B}_1+\tilde{C}_1\text{diag}\left(1+\lambda_1 \int\limits_{0}^{T} e^{(T-s)\lambda_1}ds,\ldots,1+\lambda_{n_1} \int\limits_{0}^{T} e^{(T-s)\lambda_{n_1}}ds \right),   \end{equation*}
and  $\tilde{D}$ is non-singular if $\text{det}(\tilde{B}_1+\tilde{C}_1 e^{TJ})\neq 0$.

Suppose that the matrix $\tilde{D}$ is non-singular. Then equation \eqref{alg_syst} has the unique solution $\tilde{\mu}_1=\tilde{D}^{-1}\tilde{d}$. Substituting $\tilde{\mu_1}$ into \eqref{u_dif}, we obtain $\tilde{u}_1(t)$ and hence the first components of the unique solution $(\tilde{\mu},\tilde{u}(t))$ of the boundary value problem with parameter \eqref{difeq}-\eqref{BC_trans}: 
\begin{equation}\label{sol_1}
    (\tilde{\mu}_1,\tilde{u}_1(t))=
    \Big(\tilde{D}^{-1}\tilde{d},~~ \int\limits_{0}^t e^{(t-s)J}dsJ\tilde{D}^{-1}\tilde{d}+\int\limits_{0}^t e^{(t-s)J}\tilde{f}_1(s) ds\Big).
\end{equation}

As previously stated, the second components of $(\tilde{\mu},\tilde{u}(t))$ are determined by \eqref{mu2} and \eqref{u_alg}:
\begin{equation}\label{sol_2}
    (\tilde{\mu}_2,\tilde{u}_2(t))=
    \Big(-\sum\limits_{i=0}^{\nu-1}N^i\tilde{f}_2^{(i)}(0),~~ -\sum\limits_{i=0}^{\nu-1}N^i[\tilde{f}_2^{(i)}(t)-\tilde{f}_2^{(i)}(0)]\Big).
\end{equation}

Thus, taking into account the equivalence of the original problem \eqref{dae}, \eqref{BC} and problems with parameter \eqref{ivp1}-\eqref{BC_u} and \eqref{difeq}-\eqref{BC_trans}, we can summarize our results.

\begin{theorem}\label{th1}
Let $(E,A)$ be a regular pair of square matrices and let $P$ and $Q$ be nonsingular matrices which transform \eqref{dae} to Weierstrass canonical form \eqref{WCF}. Furthermore, let $\nu=\textnormal{ind}(E,A)$ and $f\in C^{\nu}([0,T],\mathbb{R}^n)$.
Then the boundary value problem with parameter \eqref{difeq}-\eqref{BC_trans}, where $B=\tilde{B}Q^{-1}$, $C=\tilde{C}Q^{-1}$, and $d=(d_1,0)^T$, has a unique solution  if and only if:
\begin{enumerate}
    \item  the matrix $$\tilde{D}=\tilde{B}_1+\tilde{C}_1+\tilde{C}_1 J\int\limits_0^T e^{(T-s)J}ds$$ is non-singular;
    \item $    \tilde{\mu}_2=-\sum\limits_{i=0}^{\nu-1}N^i\tilde{f}_2^{(i)}(0).$
      
\end{enumerate}
 
\end{theorem}

The unique solution $(\tilde{\mu}, \tilde{u}(t)=\begin{bmatrix}
(\tilde{\mu}_1,\tilde{u}_1(t))\\(\tilde{\mu}_2,\tilde{u}_2(t))\end{bmatrix}$    of the boundary value problem with parameter \eqref{difeq}-\eqref{BC_trans} is determined by \eqref{sol_1} and \eqref{sol_2}.

\begin{theorem}
Under the assumptions of Theorem \ref{th1} the boundary value problem \eqref{dae}, \eqref{BC} with $B=\tilde{B}Q^{-1}$, $C=\tilde{C}Q^{-1}$, and $d=(d_1,0)^T$ has a unique solution 
\begin{equation*}
x(t)=Q(\tilde{\mu}+\tilde{u}(t)).    
\end{equation*}
\end{theorem}

\begin{remark}
    We can also apply our method to initial value problems for linear differential-algebraic equations with constant coefficients. Indeed, if in the boundary condition \eqref{BC} we replace $B$ and $C$ by the identity matrix $I$ and zero matrix of order $n$, respectively, we get an initial condition 
\begin{equation}\label{ivp_ic}x(0)=d,\quad d\in\mathbb{R}^n.
\end{equation} 

Then, the parameter $\tilde{\mu}$ defined as $\tilde{\mu}=Q^{-1}\mu$, where $E\mu=Ex(0)$, satisfies the equation 
\begin{equation*}
    Q\tilde{\mu}=d.
\end{equation*}
The matrix $Q$ is non-singular by assumption, so the initial value problem \eqref{dae},\eqref{ivp_ic} has a unique solution whenever       
\begin{equation*}
    d_2=\sum\limits_{i=0}^{\nu-1}N^i\tilde{f}_2^{(i)}(0).
\end{equation*}
\end{remark}

\begin{remark}
We apply the method of parameterization to problem \eqref{dae}, \eqref{BC} under assumption that the matrix pair $(E,A)$ is regular. The matrix $E$, as well as the matrix $A$, can be either singular or non-singular. Note, however, that the proposed method is not applicable if $E=O$ (when equation \eqref{dae} is purely algebraic). This is due to the choice of the parameter $\mu$: $E\mu=Ex(0)$. 
\end{remark}

\bibliographystyle{amsplain}

\end{document}